\documentclass[a4paper,11pt,reqno]{amsart}

\usepackage{amsfonts,amsthm,amssymb,amsmath,amscd}

\usepackage[utf8]{inputenc}
\usepackage[T1]{fontenc}
\usepackage[active]{srcltx}
\usepackage{multicol}
\usepackage[bookmarks=true,hyperindex,pdftex,colorlinks,citecolor=blue]{hyperref}

\usepackage[pdftex]{color}

\usepackage{tcolorbox}
\usepackage[english]{babel}
\usepackage{xcolor}

\newtheorem{Theo}{Theorem}[section]
\newtheorem{Prop}[Theo]{Proposition}

\newtheorem{Lemm}[Theo]{Lemma}

 \theoremstyle{definition}
\newtheorem{Defi}[Theo]{Definition}

\numberwithin{equation}{section}

\newcommand{\Sim}{\operatorname{Sym}}

\newcommand{\conv}{\operatorname{conv}}
\newcommand{\aconv}{\operatorname{aconv}}

\newcommand{\cconv}{\overline{\mathrm{conv}}}

\newcommand{\Real}{\operatorname{Re}}

\newcommand{\SCD}{\operatorname{SCD}}

\newcommand{\spn}{\operatorname{span}}

\def\N{\mathbb{ N}}
\def\R{\mathbb{ R}}
\def\C{\mathbb{C}}
\def\T{\mathbb{T}}

\newcommand{\eps}{\varepsilon}

\begin{document}

\title{Operations with slicely countably determined sets}

\author[V.~Kadets]{Vladimir Kadets}
\address[Kadets]{\newline  School of Mathematics and
  Informatics,\newline 
  V.~N.~Karazin Kharkiv National University,\newline  
pl.~Svobody~4, 61022 Kharkiv, Ukraine.\newline
\href{http://orcid.org/0000-0002-5606-2679}{ORCID:
  \texttt{0000-0002-5606-2679}}} \email{v.kateds@karazin.ua}

\author[A.~P\'erez]{Antonio P\'erez}
\address[P\'erez]{\newline Departamento de Matem\'{a}ticas,\newline
  Universidad de Murcia,\newline 
30100 Espinardo (Murcia), Spain.\newline
\href{http://orcid.org/0000-0001-8600-7083}{ORCID:
  \texttt{0000-0001-8600-7083} }} \email{antonio.perez7@um.es} 

 \author[D.~Werner]{Dirk Werner}
\address[Werner]{\newline Department of Mathematics,\newline  Freie
  Universit\"at Berlin, \newline 
Arnimallee~6, \,  D-14\,195~Berlin, Germany.\newline
\href{http://orcid.org/0000-0003-0386-9652}{ORCID:
  \texttt{0000-0003-0386-9652} }} 
\email{werner@math.fu-berlin.de}

\thanks{The research of the first author is done in the framework of the
  Ukrainian Ministry of Science and Education Research Program
  0115U000481, and it was partially done during his stay in Murcia
  under the support of MINECO/FEDER project MTM2014-57838-C2-1-P,
  partially during his visits to the University of Granada which were
  supported by the Spanish MINECO/FEDER project MTM2015-65020-P, and
  partially during his visit to Freie Universit\"at Berlin in the
  framework of a grant from the {\it Alexander-von-Humboldt
    Stiftung}. The second author was partially supported by the
  MINECO/FEDER project MTM2014-57838-C2-1-P and a Ph.D.\ fellowship of
  ``La Caixa Foundation''.}

\maketitle

{\centering  Dedicated to the memory of Pawe\l\ Doma\'nski
\par}

\begin{abstract}
The notion of slicely countably determined (SCD) sets was introduced
in 2010 by A.~Avil\'{e}s, V.~Kadets, M.~Mart\'{i}n,
J.~Mer\'{i} and V.~Shepelska.
We solve in the negative some natural
questions about preserving being SCD  by the operations of union,  intersection
and Minkowski  sum. Moreover, we demonstrate that corresponding
examples exist in every space with the Daugavet property and can be
selected to be unit balls of some equivalent norms. 
We also demonstrate that almost SCD sets need not  be  SCD, thus answering a
question posed  by  A. Avil\'{e}s et al.
\end{abstract}

\section{Introduction}

The property ``slicely countably determined'' (SCD for short) for
Banach spaces and their subsets was first considered and studied in
\cite{SCDspaces} (see \cite{SCDsets} for the complete version),
proving to have noticeable applications to Banach spaces with the
Daugavet property, numerical index one and other related properties
\cite{SCDsets,Rearrangement,Spears,SCDsum}. 

 Let us recall the basic definition and examples.  In this paper we
 use the letters $X$, $Y$ and $E$ to denote Banach spaces. By a slice
 of a subset $A \subset X$ we mean a non-empty set which is the
 intersection of $A$ with an open half-space. In other words, it is a
 set of the form  
 \[ 
 S(A, x^{\ast}, \varepsilon) = \Bigl\{ x \in A \colon  x^{\ast}(x) >
 \sup_{a \in A} \, x^{\ast}(a) - \varepsilon \Bigr\}, 
 \]
where $x^*\colon X \to \R$ is a non-zero real linear bounded  functional. In the case of $x^*=0$ the above definition also makes sense and gives the \emph{degenerate slice} $S(A, 0, \varepsilon) = A$. 

\begin{Defi} [\cite{SCDsets}] \label{defSCDsets} 
Let $X$ be a Banach space, $A \subset X$ be a bounded subset. A
sequence of non-empty subsets $U_n \subset A$  is called  \emph{determining}
if for each $B \subset X$ that intersects all the $U_{n}$,  $n \in
\N$, it holds that $A \subset \overline{\conv}{(B)}$. The set $A$  is
said to be  \emph{SCD} if there is a determining sequence of slices  of
$A$. The space $X$ is said to be an  \emph{SCD space} ($X \in \SCD$ for short)
whenever all its bounded subsets are SCD.  
\end{Defi}

Note that every SCD set is separable.  Let us further remark that in the 
definition of SCD sets one may also permit that some slices are degenerate ones.

We will use several times the following consequence of the Hahn-Banach
theorem,   remarked first for convex sets in
\cite[Proposition~2.2]{SCDsets}:  

\begin{Lemm}  \label{len-H-B-scd} 
 Let $U \subset X$  be a  bounded set. A sequence $\{V_n\,:\,n\in\N\}$
 of non-empty subsets of $U$ is determining for $U$ if and only if it
 has the following property ($*$):  every slice of $U$ contains one of
 the $V_n$. 
\end{Lemm}

\begin{proof}
Let ($*$) be fulfilled and let  $B\subset U$ intersect all the
$V_n$.  Then $B$ intersects all the slices of $U$, and then by the
Hahn-Banach separation theorem $\cconv(B)\supset U$. Now the ``only if''
part. Assume that some slice $S =  S(U, x^{\ast}, \varepsilon)$ of $U$
does not contain any of the 
$V_n$. Then $U \setminus S$ intersects all the $V_n$. But $U \setminus
S \subset  \{ x \in X \colon  x^{\ast}(x) \le \sup_{a \in U} \,
x^{\ast}(a) - \varepsilon \}$, hence  $\cconv(U \setminus S)
\subset  \{ x \in X \colon  x^{\ast}(x) \le \sup_{a \in U} \,
x^{\ast}(a) - \varepsilon \}$ which 
means that $\cconv(U \setminus S) \not\supset U$, and consequently
$\{V_n\,:\,n\in\N\}$ is not determining. 
\end{proof}

We can restrict ourselves to study bounded, closed and convex sets
because of the following result (\cite[Proposition~7.20]{Spears} and
\cite[Remark~2.7]{SCDsets}). Since we are going to use it several
times, we will sketch the proof for the readers' convenience.

\begin{Lemm} \label{Prop:SCDconvexhull}
Let $X$ be a Banach space. A bounded set $A \subset X$ is SCD  if and
only if its convex hull ${\conv}{(A)}$ is SCD, and if and only if its
closed convex hull $\overline{\conv}{(A)}$ is SCD. 
\end{Lemm}

\begin{proof}
It follows readily from
the definition that a bounded set is SCD if and only if its closure is, and
that $A$ is SCD once its convex hull is; cf.\
\cite[Remark~2.7]{SCDsets}. Suppose now that $A$ is SCD, 
and let $\{ S(A,x_{n}^{\ast}, \varepsilon_{n}) \colon n \in \N
\}$ be a family of slices determining for $A$. 
We  consider
 the following (countable) family of slices of $\conv(A)$: 
\[ 
\mathcal{S} := \left\{ S(\conv(A), x^{\ast}_{n},
  \varepsilon_{n}/k)\colon  n,k \in \N \right\}.   
\] 
Given any slice $S(\conv(A), x^{\ast}, \varepsilon)$ of $\conv(A)$,
where  $\|x^*\|=1$ without loss of generality, 
we will show that it contains an element of $\mathcal{S}$, thus
proving that $\conv(A)$ is SCD by Lemma~\ref{len-H-B-scd}. Now, for
the slice of $A$ given by $S(A,x^{\ast},\varepsilon/2)$ we know that
there is $n_{0} \in \N$ such that $S(A, x_{n_{0}},
\varepsilon_{n_{0}}) \subset S(A,x^{\ast},\varepsilon/2)$. Taking $k
\in \N$ big enough we will argue that 
\[ 
S(\conv(A), x^{\ast}_{n_{0}}, \varepsilon_{n_{0}}/k)  \subset
\conv{S(A, x^{\ast}_{n_{0}}, \varepsilon_{n_{0}})} +
\frac{\varepsilon}{2} B_{X} .
\]

To prove this inclusion we let
$r:=\sup_{a\in A} x_{n_0}^*(a)$ and $M:=\sup_{a\in A} \|a\|$, hence
also  $\sup_{a\in \conv(A)} x_{n_0}^*(a)=r$. Consider a convex
combination $a=\sum_{i=1}^n \lambda_i a_i$ of elements $a_i\in A$ such
that $ x_{n_0}^*(a)>r-\eps_{n_0}/k$ where $k$ is not yet
specified. Let $I=\{i\colon x_{n_0}^*(a_i)>r-\eps_{n_0}\}$ and
$J=\{i\colon x_{n_0}^*(a_i)\le r-\eps_{n_0}\}$. We then have 
$$
r-\frac{\eps_{n_0}}k < 
\sum_{i\in I} \lambda_i  x_{n_0}^*(a_i) + 
\sum_{i\in J} \lambda_i  x_{n_0}^*(a_i) \le
r \sum_{i\in I} \lambda_i  + \sum_{i\in J} \lambda_i (r-\eps_{n_0}),
$$
which implies 
$$
\sum_{i\in J} \lambda_i < \frac1k \qquad \mbox{and} \qquad
\Lambda := \sum_{i\in I} \lambda_i  > 1 - \frac1k .
$$
Now put $\mu_i := \lambda_i/\Lambda$ for $i\in I$ and consider the
element
$$
a' = \sum_{i\in I} \mu_i a_i \in \conv S(A, x^{\ast}_{n_{0}},
\eps_{n_0}).
$$
The estimate
$$
\|a-a'\| =  
\Bigl\| (\Lambda-1) \sum_{i\in I} \mu_i a_i +
\sum_{i\in J} \lambda_i a_i \Bigr\| \le 
|\Lambda-1| M + \sum_{i\in J} \lambda_i M <
\frac{2M}k
$$
shows that the above inclusion holds true whenever $k\ge 4M/\eps$. 

It now follows for this choice of $k$ that 
\[ 
\begin{split}
S(\conv(A), x^{\ast}_{n_{0}}, \varepsilon_{n_{0}}/k) & \subset
\conv{S(A, x^{\ast}_{n_{0}}, \varepsilon_{n_{0}})} +
\frac{\varepsilon}{2} B_{X} \\
&\subset \conv{S(A, x^{\ast},
  \varepsilon/2)} + \frac{\varepsilon}{2} B_{X}\\ 
& \subset S(\conv(A),x^{\ast} , \varepsilon/2) + \frac{\varepsilon}{2}B_{X}.
\end{split}
\]
Since trivially $S(\conv(A), x^{\ast}_{n_{0}}, \varepsilon_{n_{0}}/k)  \subset
\conv(A)$, we finally get 
\[
\begin{split}
S(\conv(A), x^{\ast}_{n_{0}}, \varepsilon_{n_{0}}/k) &\subset 
\Bigl( S(\conv(A),x^{\ast} , \varepsilon/2) +
\frac{\varepsilon}{2}B_{X} \Bigr) \cap \conv(A) \\
&\subset 
S(\conv(A), x^{\ast}, \varepsilon). \qedhere
\end{split}
\]
\end{proof}

 In the case of convex sets a well-known result of Bourgain lets us
 replace  the sequence of
 slices by a sequence of relatively weakly open subsets or even by a
 sequence of convex combinations of slices in the definition of an SCD
 set \cite[Proposition~2.18]{SCDsets}.  
A non-convex set can be not SCD,
 and still possess a determining sequence of relatively weakly open
 subsets; this will be proved in 
Proposition~\ref{Prop-non-conv-det-ex} below.

\begin{Prop} \label{suff-cond-scd}
 The following conditions are sufficient for a convex, bounded and
 separable subset $A \subset X$ to be SCD, see
 \cite[Section~2]{SCDsets} for details: 
\begin{enumerate}
\item[(i)] $A$ is Asplund, i.e., $(X^{\ast}, \rho_{A})$ is separable
  where $\rho_{A}(x^{\ast}) = \sup_{a \in A}{|x^{\ast}(a)|}$ for each
  $x^{\ast} \in X^{\ast}$. 
\item[(ii)] $A$ is \emph{ huskable}, i.e., $A$ is the closed convex hull of
  all $a \in A$ satisfying that for each $\varepsilon > 0$ there is a
  relatively weakly open set  $W \subset A$  with diameter less 
  than $\varepsilon$ containing $a$  (immediate consequence of
  \cite[Theorem~2.19]{SCDsets}). In particular, this happens if $A$ is
  dentable.  
\item[(iii)] $A$ is strongly regular, i.e., every convex subset $L
  \subset A$ has convex combinations of slices of arbitrarily small
  diameter. 
\item[(iv)] $(A, \sigma(A, X^{\ast}))$ has a countable
  $\pi$-basis, that is, a countable family of
  relatively weakly open non-empty 
  subsets such that each relatively weakly open non-empty subset of $A$
  contains a member of that family. This is the case, in particular, if
  $A$ does not 
  contain $\ell_{1}$-sequences. 
\end{enumerate}
\end{Prop}

As a consequence of the previous examples, if $X$ is a separable
Banach space without isomorphic copies of $\ell_{1}$ (in particular if
it is Asplund) or if it has the  convex point of continuity property
 (CPCP),  in particular if it has 
the Radon-Nikodym property (RNP), then $X \in \SCD$.  

Another class of examples \cite[Theorem~3.1]{Rearrangement}: the unit
ball of every space with a 1-unconditional basis  is SCD. It is an
open question whether every Banach space with an unconditional basis
is an SCD space. 

In order to present typical applications of SCD sets to operators in
Banach spaces,  let us introduce some definitions.   
A bounded linear operator $T\colon X \to X$  satisfies the
\emph{Daugavet equation} if   
$$
\|I +  T\| = 1 + \|T\|,
$$
and satisfies the \emph{alternative Daugavet equation} if 
$$
\max \{\|I + \theta T\| : |\theta| = 1\} = 1 + \|T\|.
$$
A Banach space has the  \emph{Daugavet property} if every rank-1
operator $T\colon X \to X$ satisfies the Daugavet equation, and it
possesses  the \emph{alternative Daugavet property} if every rank-1
operator $T\colon X \to X$ satisfies the alternative Daugavet
equation.  Typical examples of spaces possessing the Daugavet property
(and consequently the alternative Daugavet property) are  $C[0,1]$ and
$L_{1}[0,1]$.  Typical examples of spaces not possessing the Daugavet
property but nevertheless having the alternative Daugavet property are
$c_0$ and $\ell_1$. See \cite{KadSSW} and \cite{MO}. 

Theorem 4.4 of  \cite{SCDsets} says, in particular, that if $X$
possesses the alternative Daugavet property and $B_X$ is SCD, then
{\bf every} bounded linear operator on $X$ satisfies the alternative
Daugavet equation. Theorem~5.3 and Proposition~5.8 of \cite{SCDsets}
say that if  $X$ possesses the (alternative) Daugavet property and
$T(B_X)$ is SCD, then the operator $T$ satisfies the (alternative)
Daugavet equation. More applications in the same vein can be found in
\cite[Section~3]{Bosenko}, \cite[Theorems~3.4 and~3.7]{LipSlices}  and
\cite[Sections~3,~7]{Spears}.

The number of known examples of separable Banach spaces $X$ that are  not 
SCD  is limited to those having the Daugavet property.  In these
spaces the unit ball satisfies following ``anti-SCD'' condition, see
\cite[Example~2.13]{SCDsets}. 

\begin{Lemm} \label{daug-not-scd}
Let $X$ be a Banach space with the Daugavet property. Then, for every
sequence of slices $(S_{n})$ of $B_{X}$ and every $x \in S_{X}$ there
is a set $B=\{x_n\colon n \in \N\}$ with $x_{n} \in S_{n}$, $n = 1,2,
\dots$,  such that  $x  \notin \spn{B}$. 
\end{Lemm}

It is an open question whether every separable Banach space $X \notin
\SCD$  is isomorphic to a space with the Daugavet property.

When studying applications of the property SCD, it was pointed out in
\cite[Remark~4.5 and~5.4]{SCDsets} that some of them (in particular
the above-mentioned  \cite[Theorem~4.4, Theorem~5.3]{SCDsets}) hold if
in the definition of an SCD set we replace the convex hull by the
absolutely convex hull, introducing the next presumably weaker
condition: 

\begin{Defi}
A bounded set $A \subset X$ of a Banach space $X$ is called
\emph{almost SCD} (aSCD, in short) if there is a sequence of slices
$S_{n}$ of $A$ satisfying that for each $B \subset X$ with $B \cap
S_{n} \neq \emptyset$ for every $n \in \N$, it holds that $A \subset
\overline{\aconv}{(B)}$. 
\end{Defi}

It is asked in \cite[Question~7.5]{SCDsets} whether the  classes of
aSCD sets and SCD  sets coincide. In the present paper we demonstrate
that  the properties  aSCD and SCD are not equivalent for general
bounded closed convex sets, but in the case that is most important for
the applications, namely the case of balanced bounded closed convex
sets, the equivalence holds true. We also solve in the negative
natural questions about preserving SCD  by the operations of union,
intersection and Minkowski  sum.

The main part of the paper consists  of four sections. At the
beginning of Section~\ref{sec2} we  construct a set $A$, whose
properties will be the base of all the remaining examples (the letter
$A$ will be fixed afterwards for that special set). Then, we present
the  promised examples for the intersection of SCD sets (which will be
$A$ and $-A$), after that for the Minkowski sum  (which will again be
$A$ and $-A$),  and finally for the union (some shifts of   $A$ and
$-A$). In fact, we demonstrate the existence of such examples in every
space with the Daugavet property. The examples constructed in
Section~\ref{sec2}  are not centrally symmetric, which is not entirely
satisfactory, because in all the applications mentioned above only
sets symmetric with respect to zero  appear.  Section~\ref{sec3}
is devoted to the symmetrization of our examples,
after which one can see that the  operations of Minkowski sum, union
and intersection do not preserve the property SCD even if the sets in
question are unit balls of some equivalent norms. In Section~\ref{sec4}  we
give an example of an  aSCD set that is not  SCD, which will be the
union of $A$ with a specially constructed subset of  $-A$. The last
short section lists some open problems about SCD sets. 

Since in the definition of a slice and, consequently, in the
definition of an SCD space only real scalars are used, below, if the
contrary  is not stated explicitly,  {\bf we will consider only real
  Banach spaces}.  
We have already used without explanation some standard Banach space
notation like $B_X$, $S_X$ or $X^*$ for the closed unit ball, unit
sphere and the dual space respectively. All unexplained notation below
(if any) is also standard and can be found in every Banach space
textbook, for example in \cite{montesinosAnalysis}.


\section{The promised examples}\label{sec2}

\subsection{The intersection of SCD sets}

 The examples which we are going to present in this paper will be
 constructed in an arbitrary Banach space $X$ with the Daugavet
 property.  According to \cite[Theorem~4.5]{KadSW2} $X$ contains a
 separable subspace with the Daugavet property, so without loss of
 generality we assume that $X$ itself is separable. Fix a
 one-codimensional closed subspace $E \subset X$. According to
 \cite[Theorem~2.14]{KadSSW} $E$ also has  the Daugavet property, so
 $B_E$ enjoys the property from Lemma~\ref{daug-not-scd}, consequently
 $B_E$ is neither SCD, nor aSCD. The aim of the construction below is
 to include $B_E$ into an SCD set $A \subset X$ in such a way that
 $B_E$ lies in the boundary of $A$. This construction will be used in
 all the examples presented in this paper. 

Recall that a space $Y$ is called \emph{locally uniformly rotund} or
\emph{locally uniformly convex} (LUR 
for short) if for every $y \in S_Y$ and every sequence $(y_{n})$ in
$B_{Y}$ the condition $\|y + y_n\| \to 2$ implies that $\|y - y_n\|
\to 0$. In a LUR space $Y$  every point $y$ of the unit sphere $S_{Y}$
is \emph{strongly exposed}, that is, there is $y^{\ast} \in
S_{Y^{\ast}}$ with $y^{\ast}(y) = 1$  such that every sequence
$(y_{n})$ in $B_{Y}$ with $y^{\ast}(y_{n}) \rightarrow y^{\ast}(y)$
satisfies that  $\|y - y_n\| \to 0$.  It is a classical result by
M.~Kadets (\cite{LUR}, see also  \cite[p.~383,
Theorem~8.1]{montesinosAnalysis})  that every separable Banach space
admits an equivalent LUR norm. 

In fact, there is  an equivalent LUR norm $\varphi \colon E \to [0, +
\infty)$ such that $\frac12\|x\| \le \varphi (x) \le \|x\|$ for all $x
\in E$. Then for every  $t > 0$ the formula  
\[ \| x\|_{t} = \sqrt{\| x\|^{2} + t^2 \varphi(x)^2 }, \qquad x \in E  \]
defines an equivalent LUR  norm on $E$ \cite[Chapter~2, p.~53,
beginning of Section~2]{DeGoZi}  satisfying that 
\begin{equation} 
\label{equa:equivalentNorms}
\| x\| \le \| x \|_{t} \le \sqrt{1 + t^2} \: \| x\|. 
\end{equation}
In particular every point of the unit sphere $S_{(E, \|\cdot \|_{t})}$
is strongly exposed. If  $t = 0$, then we get the original norm on
$E$, i.e., $ \| x \|_{0} = \| x \|$.  We are going to use the notation
$\| \cdot \|_{t}^{\ast}$ for the norm of $(E, \| \cdot
\|_{t})^{\ast}$. In the case of $t=0$, where $\| \cdot \|_{0}$ is just
the original norm $\| \cdot \|$, we will write $\|
y^{\ast}\|_{0}^{\ast} = \| y^{\ast}\|$. 

We now construct the set which  plays the fundamental role in all our
counterexamples. Let  $e_0 \in X \setminus E$ be a fixed element of
norm 1. Then  $X = E \oplus \spn e_0$. 
In the sequel we will use notation $x \oplus t$ in order to denote an
element of the form $x + te_0$, where $x \in E$, $t \in \R$. We will
also consider the following equivalent norm on $X$: $\|x \oplus
t\|_{\infty} = \max\{\|x\|, |t|\}$. Remark that the dual space to our
$X = E \oplus \spn e_0$ can be represented as the set of formal
expressions $y^{\ast} \oplus \lambda$, $y^{\ast} \in E^*$, $ \lambda
\in \R$, that act on elements of $X$ by the natural rule $\langle
y^{\ast} \oplus \lambda, x \oplus t\rangle = y^{\ast}(x) + \lambda
t$.

\begin{Prop} \label{prop:counterexample}
The subset 
\begin{equation}\label{equa:mainSet} 
A := \{ x \oplus t \in X \colon \|x\|_{t}^{2} + 3 t^{2} \le 1, \  t
\ge 0 \}  \subset X  
\end{equation}
has the following properties:
\begin{enumerate}
\item[(a)] Every element  $x \oplus t \in A$ satisfies $ t \in
  \left[0, \frac{1}{\sqrt 3}\right]$  and  $\|x\| \le \sqrt{1 -  3
    t^{2}}$, 
in particular $A$ is  bounded.
\item[(b)] Every element $x \oplus t \in X$ satisfying 
$t \in \left[0, \frac{1}{\sqrt 3}\right]$   and  $\|x\| \le
\sqrt\frac{1 -  3 t^{2}}{1 + t^2}$ 
belongs to $A$.
\item[(c)]  $A$ is closed.
\item[(d)] $A$  is convex.
\item[(e)] $A$  is SCD. 
\end{enumerate}
\end{Prop}

\begin{proof}
Conditions (a) and (b) follow immediately from \eqref{equa:mainSet}
and  \eqref{equa:equivalentNorms}. 
(c) follows from the continuity of the map $x \oplus t \mapsto \| x\|_{t}$. 
To check (d), that $A$ is convex, note that  the set can be rewritten as 
\[ A=\{ x \oplus t \in X \colon H(\| x\|,\varphi(x),t) \le 1 \} \cap
\{ x \oplus t \in X \colon t \ge 0 \} \] 
where $H(r,s,t):=r^2 + t^2 s^2  + 3t^2$. $H$ is a convex function on
$[0,1]^{3}$, indeed its Hessian matrix  
\[ 
\left( \begin{array}{ccc}
2 & 0 & 0 \\
0 & 2t^2 & 4ts \\
0 & 4ts & 6 + 2 s^2 \end{array} \right)
\] 
is positive definite on $(0,1)^{3}$, since the determinants of its
principal minors are all positive on this domain: $\Delta_{1}=2$,
$\Delta_{2} = 4 t^{2}$ and $\Delta_{3} = 12 t^{2} (1 -
s^{2})$. Furthermore, $H$ is nondecreasing in each variable when
considered defined on $[0,1]^{3}$, so for  $x_{i} \oplus t_{i} \in A$
($i=1,2$) and $0 \le \lambda \le 1$ we have that  
\[ 
\begin{split}
& H\left(\| \lambda x_{1} + (1 - \lambda) x_{2} \|, \varphi(\lambda
  x_{1} + (1 - \lambda) x_{2}),\lambda t_{1} + (1-\lambda)
  t_{2}\right) \\ 
& \le H\left(\lambda \| x_{1}\| + (1 - \lambda) \| x_{2}\|, \lambda
  \varphi(x_{1}) + (1 - \lambda) \varphi(x_{2}), \lambda t_{1} + (1 -
  \lambda) t_{2}\right)\\  
& \le \lambda \: H\left(\| x_{1}\|, \varphi(x_{1}), t_{1}\right) + (1
  - \lambda) \: H\left(\| x_{2}\|, \varphi(x_{2}), t_{2}\right) \\
& \le 1   - \lambda + \lambda = 1.  
\end{split}
\] 
Therefore $\lambda (x_{1} \oplus t_{1}) + (1 - \lambda) (x_{2} \oplus
t_{2}) \in A$. 

We finally prove (e), that $A$ is an SCD set, by showing that it is
huskable (see Proposition~\ref{suff-cond-scd}(ii)).  
To this end, denote $\tilde A = \{x \oplus t \in A : 0 < t <
1/\sqrt{3}$, $\| x\|_t^{2} = 1 - 3 t^{2}\}$. Evidently,
$\overline{\conv}{(\tilde A)} = A$, so it remains to demonstrate the
following statement:  
\begin{quote}
For every $\varepsilon > 0$ and every $x_{0} \oplus t_{0} \in \tilde
A$  there is a relatively weakly open subset of $A$ containing $x_{0}
\oplus t_{0}$ with  $\|\cdot\|_{\infty}$-diameter less than
$4\varepsilon$.  
\end{quote}
For this, let us write briefly $r_{0}:= (1 - 3t_{0}^{2})^{1/2} = \|
x_{0}\|_{t_{0}}$.  

Since $x_{0}$ is a strongly exposed point of $r_{0} \, B_{(E, \| \cdot
  \|_{t_{0}})}$, there exist $x_{0}^{\ast} \in S_{(E^{\ast}, \|
  \cdot\|_{t_{0}}^{\ast})}$ and $\beta_{0} \in (0,1)$ satisfying: 
\begin{enumerate}
\item[(i)] $x_{0}^{\ast}(x_{0})= r_{0} $.
\item[(ii)]  If $x \in r_{0} \, B_{(E, \| \cdot \|_{t_{0}})}$ and
  $x_{0}^{\ast}(x) > \beta_{0}r_{0}$, then $\| x - x_{0}\| <
  \varepsilon$. 
\end{enumerate}
 Take $\delta > 0$ small enough so that
\begin{equation}\label{equa:deltaConditions}  
x_{0}^{\ast}(x_{0}) > \beta_{0} (r_{0} + 2 \delta) \hspace{3mm} \mbox{
  and } \hspace{3mm}  \frac{2\delta}{2\delta + r_{0}} +
\frac{\delta^{2}}{2} < \varepsilon.  
\end{equation}
Consider the relatively weakly open subset $W$ of $A$ given by
\[ W:=\{ x \oplus t  \in A\colon x^{\ast}_{0}(x) > \beta_{0}(r_{0} + 2
\delta),\: |t - t_{0}| < \delta^{2}/2 \}. \] 
It is immediate that $x_{0} \oplus t_{0} \in W$. Furthermore, given $x
\oplus t \in W$ we have that $|t^{2} - t_{0}^{2}| < \delta^{2}$ and
hence 
\[ 
\begin{split}
 \| x\|_{t_{0}}  &=   \left(\| x\|^{2} + t_{0}^{2} \varphi(x)\right)^{1/2}  \\
 &=   \left(\| x\|_t^{2} + (t_{0}^{2} - t^2) \varphi(x)\right)^{1/2} \\
 & \le \left(\| x\|_{t}^{2} + |t_{0}^{2} - t^{2}|\right)^{1/2}\\
&  \le \left( 1 - 3t^{2} + |t_{0}^{2} - t^{2}| \right)^{1/2}  \\
&\le \left( 1 - 3t_{0}^{2} + 4\delta^{2} \right)^{1/2} \\
&\le  r_{0} + 2 \delta.
\end{split}
\]
The last inequality together with \eqref{equa:deltaConditions} gives that
\[ 
\left\| \frac{r_{0} \ x}{r_{0} + 2\delta} \right\|_{t_{0}} \le
r_{0} \hspace{3mm} \mbox{ and } \hspace{3mm} x_{0}^{\ast}\left(
  \frac{r_{0} \ x}{r_{0} + 2\delta} \right) > \beta_{0} r_{0}. 
\] 
By (ii) it follows that
\[ 
\varepsilon > \left\| \frac{r_{0} \ x}{r_{0} + 2\delta} -
  x_{0}\right\| \ge \| x - x_{0}\| - \| x\| \frac{2\delta}{r_{0} +
  2\delta} \ge \| x - x_{0}\| - \frac{2\delta}{r_{0} + 2\delta}, 
\]
and therefore 
\[ 
\left\| x \oplus t \, - \,x_{0} \oplus t_{0} \right\|_{\infty} =
\max{\{ \| x - x_{0}\|, |t - t_{0}|\}} < \max{\left\{\varepsilon +
    \frac{2\delta}{2\delta + r_{0}}, \frac{\delta^{2}}{2}\right\}} <
2\varepsilon. 
\]
We conclude then that the diameter of $W$ is less than $4 \varepsilon$
finishing the proof of the statement above. 
\end{proof}

Remark also that $A$ in the above Proposition has two more evident
properties: it has non-empty interior, and for every $x \oplus t \in
A$ also  $(-x) \oplus t \in A$. 

\begin{Theo} \label{prop-scd-intersect}
In every Banach space $X$ with the Daugavet property there are convex
closed bounded  SCD sets $A, D \subset X$ whose intersection $A \cap
D$ is not SCD. 
\end{Theo} 

\begin{proof}
Let $A$ and $E$ be as in Proposition~\ref{prop:counterexample}, 
 and let $D = -A$. Both sets are SCD by
Proposition~\ref{prop:counterexample} although $A \cap D = B_{E}$ is
not. 
\end{proof}


\subsection{Sum and union of SCD sets}

For  $B_1, B_2 \subset X$ we denote, as usual, by $B_1+B_2$  the
corresponding Minkowski  sum: $B_1+B_2 = \{b_1 +b_2: b_1 \in B_1$, $b_2
\in B_2\}$. The need to consider various  Minkowski  sums appears in
many instances, in particular in the applications of SCD sets to
operator theory. The first theorem of that kind appeared in
\cite[Corollary~3.9]{SCDsets}, where inheritance of the property  SCD
of two \textbf{spaces} by their  \textbf{direct} sum was
demonstrated. The next step (that has important applications) was done
in  \cite[Theorem~2.1]{SCDsum}: the  \textbf{direct} sum of two
\textbf{hereditarily} SCD sets is a hereditarily SCD set again
(\emph{hereditarily SCD} means that all subsets are SCD). In the same
paper it was demonstrated that in the statement of the latter result
the direct sum cannot be substituted by the Minkowski sum. Namely, in
\cite[Corollary~2.2]{SCDsum} it is demonstrated that the Minkowski sum
of two hereditarily SCD sets need not   be hereditarily
SCD. Unfortunately, the statement of  \cite[Corollary~2.2]{SCDsum} as
it appeared in the paper, viz.\ ``The sum of two hereditarily SCD sets
need not be an SCD set,'' contains a misleading misprint: the second
word ``hereditarily'' is missing.  The construction in
\cite[Corollary~2.2]{SCDsum} consists of two separable RNP subsets $U,
V \subset \ell_1 \oplus_\infty C[0,1]$ such that $U+V \supset
B_{C[0,1]}$ which makes  $U+V$ not  hereditarily SCD. Nevertheless,
\cite[comments after Prop.~1.7]{Schach},   $\overline{U+V}$ is the
closed convex hull of its strongly exposed points, so $U+V$ is
SCD. The authors noted that painful misprint only some years after the
publication (see Editor's comment to Zentralblatt review
\href{https://zbmath.org/?q=an:1210.46010}{Zbl 1210.46010}), and since
then it has remained an open question whether the statement with the
misprint is incidentally also correct.  In this subsection we answer a
related question, demonstrating that the Minkowski sum of two SCD sets
need not be SCD. We also give  an analogous result about unions of SCD
sets. 

At first, remark the following easy properties:

\begin{Lemm} \label{lem-sum-slice} 
Let $B_1, B_2 \subset X$ be non-empty bounded sets and let $x^{\ast}
\in   X^{\ast}$, $\eps > 0$. We then  have the
following properties: 
\begin{enumerate}
\item[(i)] $S(B_1,  x^{\ast}, \varepsilon /2) + S(B_2,  x^{\ast},
  \varepsilon/2) \subset S(B_1 +B_2,  x^{\ast}, \varepsilon)$. 
\item[(ii)] If $a \in B_1$, $ b \in B$ satisfy that $a+b \in S(B_1
  +B_2,  x^{\ast}, \varepsilon)$,  then $a \in S(B_1,  x^{\ast},
  \varepsilon)$ and $b \in S(B_2,  x^{\ast}, \varepsilon)$. 
\end{enumerate}
\end{Lemm}

The above Lemma and   Lemma~\ref{len-H-B-scd}  imply the following result.

\begin{Lemm}\label{Lemm:sumSCDcharact}
Let $B_1, B_2 \neq \emptyset$ be bounded subsets of a Banach space $X$. Then 
the following assertions are equivalent:
\begin{enumerate}
\item[(a)]  
$B_1 +B_2$ is SCD.  
\item[(b)] 
There exists a countable family  $(x_{n}^{\ast}, \varepsilon_{n}) \in
 X^{\ast}  \times (0,+\infty)$ satisfying that for
every $(x^{\ast}, \varepsilon) \in    X^{\ast} \times
(0,+\infty)$ there is an $m \in \N$ such that 
\begin{equation} 
\label{equa:sum of slices}
 S(B_1,  x_{m}^{\ast}, \varepsilon_{m}) \subset S(B_1,  x^{\ast},
 \varepsilon) \hspace{3mm} \mbox{ and } \hspace{3mm} S(B_2,
 x_{m}^{\ast}, \varepsilon_{m}) \subset S(B_2,  x^{\ast},
 \varepsilon). 
\end{equation} 
\end{enumerate}
\end{Lemm}

\begin{proof}
(a) $\Rightarrow$ (b): Let $S_n = S(B_1 +B_2,  x^*_n, 2\eps_n)$ with
$(x_{n}^{\ast}, \varepsilon_{n}) \in    X^{\ast}
\times (0,+\infty)$,   $n \in \N$,  be  
slices of $B_1 +B_2$ which form a  determining sequence. Let us
demonstrate that  $(x_{n}^{\ast}, \varepsilon_{n})$ form the sequence
we need for (b). Indeed, according to   Lemma~\ref{len-H-B-scd}  for
every $(x^{\ast}, \varepsilon) \in   X^{\ast}  \times
(0,+\infty)$ there is $m \in \N$ such that $S_m \subset S(B_1 +B_2,
x^*, \eps)$, and by (i) of Lemma~\ref{lem-sum-slice} also $ S(B_1,
x^*_m, \eps_m) +  S(B_2,  x^*_m, \eps_m)   \subset S(B_1 +B_2,  x^*,
\eps)$. An application of  (ii) of Lemma~\ref{lem-sum-slice} gives us
the desired inclusions \eqref{equa:sum of slices}. 

(b) $\Rightarrow$ (a):  Assume  $(x_{n}^{\ast}, \varepsilon_{n}) \in
 X^{\ast}  \times (0,+\infty)$ are from (b), and let
us demonstrate that the slices  $S_n = S(B_1 +B_2,  x^*_n, \eps_n)$
form  a determining sequence of slices for $B_1 +B_2$. Fix a slice
$S(B_1 +B_2,  x^*, 2\eps)$ with $x^* \in X^* \setminus \{0\}$, $\eps >
0$  and, using (b), select  an $m$ for which \eqref{equa:sum of
  slices} takes place. We are going to demonstrate that $ S(B_1 +B_2,
x^*_m, \eps_m) \subset  S(B_1 +B_2,  x^*, 2\eps)$. Indeed, let $x \in
S(B_1 +B_2,  x^*_m, \eps_m) $ be an arbitrary element. Then it is of
the form $x = a+b$, $a \in B_1$, $b \in B_2$, and, by  (ii) of
Lemma~\ref{lem-sum-slice},  $a \in S(B_1,  x^*_m, \eps_m) $,  $b \in
S(B_2,  x^*_m, \eps_m) $.  It remains to apply  (i) of
Lemma~\ref{lem-sum-slice}:  
\[ 
\begin{split}
x = a+b &\subset S(B_1,  x^*_m, \eps_m)  + S(B_2,  x^*_m, \eps_m)  \\ 
&\subset S(B_1,  x^*, \eps)  + S(B_2,  x^*, \eps)  \subset S(B_1 +B_2,
x^{\ast}, 2\varepsilon).  \qedhere
\end{split}
\]         
\end{proof} 
 
The above lemma leads to the following result.

\begin{Theo} \label{theo:SumSCDimplies-summandsSCD}
Let $B_{1}, B_2$ be non-empty bounded subsets of a Banach space $X$
such that $B_1 + B_2$ is SCD. Then,  $B_{1}$ (and so also $B_{2}$) is
SCD.  
\end{Theo}

\begin{proof}
Let $(x_{n}^{\ast}, \varepsilon_{n}) \in    X^{\ast}
\times (0,+\infty)$ be the family from (b) of
Lemma~\ref{Lemm:sumSCDcharact}, then the slices $S(B_1, x_{m}^{\ast},
\varepsilon_{m})$ form a determining sequence for $B_1$. 
\end{proof} 

 The next proposition explains some difficulties that arise when one
 has to demonstrate that a non-convex set is SCD. 

\begin{Prop} \label{Prop-non-conv-det-ex} 
There are non-convex non-SCD sets containing 
 a determining sequence of relatively weakly open subsets. Such
 examples exist in every Banach space with the Daugavet property. 
\end{Prop}

\begin{proof}
Let $X$ be a space with the Daugavet property (as before it can be
assumed separable), $E$ be a 1-codimensional closed subspace. Then $X$
is isomorphic to $E \oplus_\infty \R$. Take a sequence $(x_n)$
in the unit ball 
of $E$ such that both subsequences $(x_{2n})_{n \in \N}$ and $(x_{2n-1})_{n \in \N}$ are dense and a sequence of $ t_n \in (0,1)$ such that $t_{2n} \to 0$
and $t_{2n+1} \to 1$. The set in question will be the following subset
of $E \oplus_\infty \R$:
$$
U = \{x_n \oplus t_n: n \in \N\}.
$$
This set is quickly seen to be discrete in the weak topology, 
so  $\left\{\{x_n \oplus t_n\}\colon  n \in \N \right\}$ is the requested
determining sequence 
of relatively weakly open subsets. On the other hand the closed convex hull
of $U$ equals $B_E \oplus [0, 1]$ which, according to
Theorem~\ref{theo:SumSCDimplies-summandsSCD},  is not SCD because the
unit ball of $E$  is not SCD. 
\end{proof}

Now we are ready for the first main result of the subsection
demonstrating that the converse to
Theorem~\ref{theo:SumSCDimplies-summandsSCD} is not true. 

\begin{Theo}\label{Prop:sumNotSCD}
In every Banach space $X$ with the Daugavet property there are convex
closed bounded  SCD sets $A, D \subset X$ whose sum $A+D$ is not SCD. 
\end{Theo}

\begin{proof}  We will use the same sets $A, D \subset X$  as in
  Theorem~\ref{prop-scd-intersect}: 
\begin{align*} 
 A &= \{ x \oplus t \colon \| x\|_{t} \le 1 - 3 t^{2}, \ t \ge 0 \}, \\
 D &= -A = \{ x \oplus t \colon \| x\|_{t} \le 1 - 3 t^{2}, \ t \le 0 \},
\end{align*}
whose intersection is $B_E$. It has  already been shown in
Proposition~\ref{prop:counterexample} 
that $A$ and $D$ are SCD.  

To see that the sum $A+D$ is not SCD we will argue by
contradiction. If we assume that $A+D$ is SCD then we could find a
countable family $(x_{n}^{\ast}, \varepsilon_{n}) \in S_{X^{\ast}}
\times (0,1)$ as in Lemma~\ref{Lemm:sumSCDcharact}. Notice that we can
write $x_{n}^{\ast}=y_{n}^{\ast} \oplus \lambda_{n} \in X^{\ast} =
E^{\ast} \oplus \mathbb{R}$. Since $B_E$ is not SCD  we can find
$y^{\ast} \in S_{E^{\ast}}$ and $\delta \in (0,1)$ such that for every
$ n \in \N$ 
\begin{equation}\label{equa:auxSumSCDcharact1} 
S(B_{E}, y_{n}^{\ast}, \varepsilon_{n}) \not\subset S(B_{E}, y^{\ast}, \delta).
\end{equation}
Considering the element $x^{\ast}=y^{\ast} \oplus 0 \in B_{X^{\ast}}$
we have that there is $k \in \N$ satisfying 
 \[
S(A, x_{k}^{\ast}, \varepsilon_{k}) \subset S(A, x^{\ast}, \delta)
\hspace{3mm} \mbox{ and } \hspace{3mm} S(D, x_{k}^{\ast},
\varepsilon_{k}) \subset S(D, x^{\ast}, \delta) 
\]
from which  it easily follows that
\begin{equation}\label{equa:auxSumSCDcharact2}
S(A, x_{k}^{\ast}, \varepsilon_{k}) \cup S(D, x_{k}^{\ast},
\varepsilon_{k}) \subset \{ x \oplus t \in X \colon x \in S(B_{E},
y^{\ast}, \delta) \}.   
\end{equation}
We now claim that 
\begin{equation}\label{equa:auxSumSCDcharact3}
S(B_{E}, y_{k}^{\ast}, \varepsilon_{k})  \subset S(A, x_{k}^{\ast},
\varepsilon_{k}) \cup S(D, x_{k}^{\ast}, \varepsilon_{k}) 
\end{equation}
which together with \eqref{equa:auxSumSCDcharact2} leads to 
\[ S(B_{E}, y_{k}^{\ast}, \varepsilon_{k}) \subset S(B_{E}, y^{\ast},
\delta), \] 
contradicting \eqref{equa:auxSumSCDcharact1} and finishing the
proof. To show the validity of the claim we distinguish two
cases. Assuming that $\lambda_{k} \le 0$ we get that
$\sup\{x_{k}^{\ast}(a): a \in A\} = \sup\{y_{k}^{\ast}(x): x \in B_E\}
= \| y_{k}^{\ast}\|$. As a consequence,  
$S(B_{E}, y_{k}^{\ast}, \varepsilon_{k}) \subset S(A, x_{k}^{\ast},
\varepsilon_{k})$. 
On the other hand, if $\lambda_{k} \ge 0$ then analogously
$\sup\{x_{k}^{\ast}(d): d \in D\} = \sup\{y_{k}^{\ast}(x): x \in
B_E\}$  and therefore 
$S(B_{E}, y_{k}^{\ast}, \varepsilon_{k}) \subset S(D, x_{k}^{\ast},
\varepsilon_{k})$. 
\end{proof}

Nevertheless, for the direct sum of SCD sets the situation remains
simple (for hereditarily SCD sets that was proved earlier in
\cite[Theorem~2.1]{SCDsum}). 

\begin{Theo} \label{Obse:summandsSCDimpliesDirectSumSCD}
Let $B_{1} \subset X_1$, $B_{2} \subset X_2$ be bounded subsets of a
Banach space $X = X_1 \oplus X_2$, and suppose that $B_1, B_2$ are
SCD.  Then $B_1 + B_2$ is SCD. 
\end{Theo}

\begin{proof}
 Let $S(B_i,y_{n,i}^{\ast}, \delta_{n,i})$,  $y_{n,i}^{\ast} \in
 X_i^{\ast}$, $n \in \N$, form  determining sequences
 of slices for $B_i$,  $i = 1,2$. Then, the collection of functionals
 $x_{n,m}^{\ast} = y_{n,1}^{\ast} \oplus y_{m,2}^{\ast}$ and
 corresponding  $\eps_{n,m} = \min\{\delta_{n,1}, \delta_{m,2}\}$
 will be  a countable family that satisfies condition~(b) of
 Lemma~\ref{Lemm:sumSCDcharact}. 
\end{proof}

And now for the last of the promised main examples of the subsection.

\begin{Theo} \label{Prop-non-conv-union} 
In every Banach space $X$ with the Daugavet property there is an SCD
set $B$ such that $ B \cup (-B)$ is not SCD. 
\end{Theo}

\begin{proof}
We follow the notation of Proposition~\ref{prop:counterexample}. Let
$\alpha := 1/(2 \sqrt{3})$ and $B := A -  \alpha e_0 =\{x \oplus (t -
\alpha): x \oplus t \in A\}$, where $A$ is the set in
\eqref{equa:mainSet}. We claim that  
$$
\conv( B \cup (-B)) =  B_{E} \oplus [-\alpha, \alpha].
$$
 Indeed, it is clear that $B$ is contained in $B_{E} \oplus [-\alpha,
 \alpha]$, and so is $-B$. For the converse, use simply that $B_{E}  -
 \alpha e_0 \subset B$,  
 $B_{E}  + \alpha e_0 \subset -B$, and consequently
$$
B_{E} \oplus [-\alpha, \alpha] = \conv {(B_{E}  - \alpha e_0) \cup
  (B_{E}  + \alpha e_0)} \subset \conv( B \cup (-B)). 
$$
 Finally, if $ B \cup (-B)$ were SCD, then $B_{E} \oplus [-\alpha,
 \alpha]$ would be SCD by Lemma~\ref{Prop:SCDconvexhull}. But it was
 already remarked above that this is never the case by
 Theorem~\ref{theo:SumSCDimplies-summandsSCD}, as $B_{E}$ is not SCD
 because of the Daugavet property of $E$  \cite[Example~2.13]{SCDsets}. 
\end{proof}


\section{Symmetrization of the examples}\label{sec3}

 In the most important applications of SCD sets, the sets which appear
 are  balls and images of balls under the action of linear
 operators. So, it would be 
natural to ask whether examples demonstrating  non-stability of the
 property SCD can be constructed to be balls of some equivalent norms,
 that is, to be convex closed bounded symmetric bodies. The keyword
 here is ``symmetric'' because the examples that we have constructed above
 possess all the remaining properties of being  convex closed bounded,
 and to have non-empty interior. In this section we apply a natural
 symmetrization procedure which helps to obtain symmetric examples
 from non-symmetric ones. 
 
Let $U$ be a bounded   non-empty  subset of a Banach space $X$. By
symmetrization of $U$ we will mean the following subset of $\Sim(U)
\subset X \oplus_\infty \R$:  
$$
\Sim(U) = \aconv(U \oplus 1) .
$$

\begin{Lemm}  \label{lem-slice-hyperplane} 
 Let $U, V \neq \emptyset$ be bounded subsets such that $U$ is
 contained in a closed hyperplane $H_0$, and $V$ lies on one side of
 $H_0$  at a positive distance from $H_0$. Then, every slice of $U$ is
 at the same time a slice of $U \cup V$. 
\end{Lemm}

\begin{proof}
 Without loss of generality we can assume that $0 \in U$, and that $U,
 V \subset B_X$ (this can be done by shifting and scaling). Then there
 are an $x_0^* \in S_{X^*}$ and $\eps_0 > 0$ such that $H_0=\ker x^*_0
 \supset  U$, and $V \subset \{x \in X\colon x_0^*(x) <  - \eps_0\}$.
 Let $x^* \in S_{X^*}$ and let $S = S(U, x^*,\eps)$ be a slice of $U$.
 Denote $r = \sup_{x\in U} x^*(x) \in  [-1, 1]$ and consider for every
 $t > 0$ the functional $ x_t^* = x^* + t  x_0^*$. Since on $U$ the
 values of $ x_t^*$  and of $x^*$ are the same,  $S = S(U,
 x_t^*,\eps)$ for all $t  > 0$. We are going to demonstrate that for
 some values of $t  > 0$ the slice  $S_t = S(U \cup V, x_t^*,\eps)$ of
 $U \cup V$ also equals $S$, which will complete our proof.  

So our goal is to show that there is a $t > 0$ such that $S_t \cap V =
\emptyset$. Assume to the  contrary that for every $t > 0$ there is an
element  $v_t \in V \cap S_t$. Then 
\[
\begin{split}
1  -  t  \eps_0 &\ge  x^*(v_t) + t  x_0^*(v_t) =  x_t^*(v_t) \\
& > \sup_{x \in U \cup V} x_t^*(x) - \eps \\ 
&=  \max\Bigl\{\sup_{x \in U} x^*(x) , \sup_{x \in  V} x^*(x) + t
x_0^*(x) \Bigr\} - \eps  \\  
&\ge  \max\{r , r -  t  \eps_0 \} - \eps = r - \eps,
\end{split}
\]
which means that $t < \frac{1+\eps - r}{\eps_0}$. This is a contradiction.
\end{proof}

\begin{Lemm}  \label{lem-slice-hyperplane++} 
If under the conditions of Lemma~\ref{lem-slice-hyperplane} $U \cup V$
is SCD, then $U$ is  also an SCD set. 
\end{Lemm}

\begin{proof}
 Let $\{V_n\colon n \in \N\}$ be a determining sequence of slices of
 $U \cup V$.  Denote $N_1 = \{n \in \N: V_n \cap U  \neq
 \emptyset\}$. Then $S_n:= V_n \cap  U $, $n \in N_1$, are slices of
 $U$. We are going to demonstrate that the collection $\{S_n : n \in
 N_1\}$ is determining for  $U$, which will do the job.  Let us use
 Lemma~\ref{len-H-B-scd}.  Consider a  slice $S$ of  $U$. Then, by
 Lemma~\ref{lem-slice-hyperplane}, $S$ is at the same time a slice of
 $U \cup V$. So, there is an $n \in \N$ such that $V_n \subset S$, but
 this $n$ automatically belongs to $N_1$.  
\end{proof}

\begin{Lemm}  \label{lem-symm-gener} 
 The following conditions for a bounded  non-empty subset $U \subset
 X$  are equivalent: 
 \begin{enumerate}
 \item [(i)]
 $U$ is SCD, 
 \item[(ii)]
 $(U \oplus 1) \cup - (U \oplus 1)$   is SCD,  
 \item[(iii)]
 $\Sim(U)$ is SCD.
 \end{enumerate}
\end{Lemm}

\begin{proof}
  Taking into account that $\Sim(U) = \conv\left((U \oplus 1) \cup - (U
  \oplus 1)\right)$ the equivalence (ii) $\Leftrightarrow$ (iii)
  follows from Lemma~\ref{Prop:SCDconvexhull}. 

(i) $\Rightarrow $ (ii): Let $\{S_n\colon n \in \N\}$ be a determining
sequence of slices of $U$. Then, the $V_n := S_n \oplus 1$  form a
determining  sequence of slices of $U \oplus 1$ and the  $- V_n$  form
a determining  sequence of slices of $ - (U \oplus 1)$. By
Lemma~\ref{lem-slice-hyperplane}, $\pm V_{n}$ are also slices of  $(U
\oplus 1) \cup - (U \oplus 1)$. But then the countable collection
$\{\pm V_{n}\colon n \in \N\}$ forms a determining sequence of slices
of $(U \oplus 1) \cup - (U \oplus 1)$. Indeed,  let $V \subset  X
\oplus_\infty \R$  intersect all  $\pm V_{n}$, $n \in \N$. Then, since
$\{V_n\colon n\in\N\}$ is determining for $U \oplus 1$, we have
$\overline{\conv} (V) \supset U\oplus 1$  and since the $- V_n$  form a
determining  sequence of slices of $ - (U \oplus 1)$ we also have
$\overline{\conv} (V) \supset -( U\oplus 1)$, which completes the proof
of the implication (i) $\Rightarrow$~(ii). 

(ii) $\Rightarrow$ (i):  Applying Lemma~\ref{lem-slice-hyperplane++} we
obtain that  $U \oplus 1$ is SCD,  but $U \oplus 1$ is a shift of $U$,
so $U$ is also SCD. 
\end{proof}

The next example is based on the elementary fact that  the convex hull
of the union of symmetrized   sets $\Sim(U_1) \cup \Sim(U_2)$ is equal
to  the symmetrized union $\Sim(U_1 \cup U_2)$: 
$$
\conv\left(\Sim(U_1) \cup \Sim(U_2)\right) = \Sim(U_1 \cup U_2).
$$   
Indeed, the left hand side is convex and symmetric, contains $(U_1
\cup U_2)\oplus 1$, so contains $\Sim(U_1 \cup U_2)$. Conversely,  the
right hand side is convex, contains $\Sim(U_1)$  and $\Sim(U_2)$, so
contains the convex hull $\conv\left(\Sim(U_1) \cup \Sim(U_2)\right)$.

\begin{Theo} \label{prop-scd-intersect-symm}
In every Banach space $Y$ with the Daugavet property there are convex
closed bounded symmetric sets $\tilde B_1, \tilde B_2 \subset Y$
which are  SCD sets, but  whose union $\tilde B_1 \cup \tilde B_2$ is
not SCD. If, additionally, $Y$ is separable, then these 
$\tilde B_1, \tilde B_2 \subset Y$ can be chosen to have non-empty interior.
\end{Theo}

\begin{proof}
It is sufficient to consider the case of separable $Y$ (otherwise,
substitute it by a separable subspace  with the Daugavet
property). Let $X$ be a one-codimensional closed subspace of $Y$. Then
$X$ also has  the Daugavet property. Our $Y$  is isomorphic to  $X
\oplus_\infty \R$, so it is sufficient to construct the requested
example in $X \oplus_\infty \R$. 
Let $B$ and $-B$ be  SCD sets from Theorem~\ref{Prop-non-conv-union},
and take   $B_1 ={\Sim(B)}$, $ B_2 ={\Sim (-B)}$. Since $B$, 
$-B$ are  convex, bounded and have non-empty interior,   $B_1$
and  $ B_2$ are  convex bounded symmetric bodies which are SCD
by the previous Lemma~\ref{lem-symm-gener}. Also,
$\conv( B_1 \cup B_2)  = \Sim(B \cup (-B))$, so by the same
Lemma~\ref{lem-symm-gener}  $\conv( B_1 \cup   B_2)$   is not SCD, and
consequently $B_1 \cup 
B_2$ is not SCD.   
To finish the proof define $\tilde B_1 $ and $\tilde B_2$ to be the
closures of $B_1$ and $B_2$ and apply Lemma~\ref{Prop:SCDconvexhull}.  
\end{proof}

In order to proceed with the symmetrization of the example about the
sum of SCD sets, we first need a natural lemma.

\begin{Lemm}  \label{lem-convhull-sum} 
Let $U_1, U_2 \subset X$  be  bounded  not empty subsets. Then,
 $$  \cconv(U_1 + U_2) = \overline{\conv (U_1) +  \conv (U_2)}.  $$
\end{Lemm}

\begin{proof}
Both the right hand side and the left hand side of the equality in
question are convex closed sets, so each of them is the intersection
of all half-spaces that contain it. In other words, in order to prove
the equality it is sufficient to demonstrate that for every $x^* \in
X^*$ 
$$
 \sup x^* (\cconv(U_1 + U_2))  = \sup x^*\big( \overline{\conv (U_1) +
 \conv (U_2)}\big).  
$$
This equality is easily seen to be  true, because its  right hand side
and left hand side are both equal to $\sup x^*\left( U_1\right) + \sup
x^*\left( U_2\right)$. 
\end{proof} 

\begin{Theo} \label{prop-scd-sum-symm}
In every  Banach space $Y$ with the Daugavet property there are convex
closed bounded symmetric  SCD sets (which in the separable case can be
chosen to be bodies) $C_1, C_2 \subset Y$ 
 whose sum $C_1 + C_2$ is not SCD. 
\end{Theo}

\begin{proof}
As before, we can reduce the situation to a separable space of the
form  $X \oplus_\infty \R$, where $X$ has the Daugavet property. 
Let $B_1 := A$, $B_2 := D$ be  SCD subsets of $X$ from
Theorem~\ref{Prop:sumNotSCD} such that $B_1 + B_2$  is not  SCD, and
take   $C_1 = \overline{\Sim(B_1)}$, $C_2 =
\overline{\Sim (B_2)}$, which are closed  convex  bounded symmetric
SCD bodies. It remains to demonstrate that $C_{1} + C_{2}$
is not SCD.  Using Lemma~\ref{lem-convhull-sum} we can see that 
\[
\begin{split}
\overline{C_{1} + C_{2}} &= \overline{\conv((B_1 \oplus 1) \cup - (B_1 \oplus
1)) +  \conv((B_2 \oplus 1) \cup - (B_2 \oplus 1))} \\  
& = \cconv\left(((B_1 \oplus 1) \cup - (B_1 \oplus 1)) + ((B_2 \oplus
  1) \cup - (B_2 \oplus 1))\right). 
\end{split}
\]
According to  Lemma~\ref{Prop:SCDconvexhull}, it is sufficient to show
that the set 
\[
\begin{split}
((B_1 \oplus 1) \cup - (B_1 \oplus 1)) &+ ((B_2 \oplus 1) \cup - (B_2
\oplus 1)) \\ 
= ((B_1+B_2) \oplus 2)  &\cup  ((B_1-B_2)\oplus 0) \\
& \cup  ((B_2-B_1)\oplus 0) \cup
(-(B_1+B_2) \oplus (-2)) 
\end{split}
\]
is not SCD.  With the help of Lemma~\ref{lem-slice-hyperplane} this
can be deduced from the fact that $B_1 + B_2$  is not  SCD  exactly
the same way as the implication (ii) $\Rightarrow$ (i) of
Lemma~\ref{lem-symm-gener}, because $(B_1+B_2) \oplus 2$ lies in the
hyperplane of those elements whose second coordinate equals~$2$, and the
rest of the set lies at a distance at least~$2$ from that hyperplane. 
\end{proof} 

Before coming to the symmetrization of the non-SCD intersection
example, one more easy remark. 

\begin{Lemm}  \label{lem-convhull-monotonicity} 
Let $U_0, U_1 \subset X$  be  non-empty subsets with  $U_0 \subset
U_1 $, and let $U_1$ be convex. Then, $U_\lambda := \lambda U_1 + (1 -
\lambda) U_0$ increases when $\lambda \in [0, 1]$ increases. 
\end{Lemm}

\begin{proof}
Let $0 \le \lambda \le \mu\le1$. Then 
\[
\begin{split}
U_\mu=  \mu U_1 + (1 - \mu) U_0  &= 
\lambda U_1 + (\mu - \lambda)U_1 + (1 - \mu) U_0 \\
&\supset \lambda U_1 + (\mu - \lambda)U_0 + (1 - \mu) U_0  \\
&\supset  \lambda U_1 + (1 - \lambda) U_0 = U_\lambda. \qedhere
\end{split}
\]
\end{proof}

Also remark that if $U \subset X$  is convex, then $\Sim(U) \subset X
\oplus_\infty \R$ can be written as  
$$
\Sim(U) = \{(tu - (1-t)v)\oplus(2t-1): u,v \in U, \ t \in [0,1]\} .
$$
In other words, 
\begin{equation*}
\Sim(U) = \bigcup_{t \in [0,1]}(tU - (1-t)U)\oplus(2t-1).
\end{equation*}
This implies the following formula for the intersection of $\Sim(U_1)
\cap \Sim(U_2)$ in the case of convex $U_1, U_2 \subset X$: 
\begin{equation*} \label{formula-symm-intersect1} 
\Sim(U_1) \cap \Sim(U_2) = \bigcup_{t \in [0,1]}\left((tU_1 -
  (1-t)U_1) \cap (tU_2 - (1-t)U_2)\right)\oplus(2t-1). 
\end{equation*}

\begin{Theo}  \label{theo-symm-intersect} 
In every (separable)  Banach space $Y$ with the Daugavet property
there are convex 
closed bounded symmetric sets (bodies)  which are  SCD sets, but  whose
intersection  is not SCD.  
\end{Theo}

\begin{proof}
Again,  it is sufficient to consider a separable space of the form  $X
\oplus_\infty \R$, where $X$ has the Daugavet property.  Let $A
\subset X$ be as in Proposition~\ref{prop:counterexample}. Denote $U_1
= A$, $U_2 = -A$.  We are going to demonstrate that  $\Sim(U_1) ,
\Sim(U_2) \subset X \oplus_\infty \R$ are the requested  non-empty
bounded  convex symmetric SCD bodies such that  $W:=\Sim(U_1) \cap
\Sim(U_2)$ is not SCD.  

Each element of $X$ is of the form $e + te_0$, $e \in E$, $t \in \R$,
and in  order to avoid misunderstanding we will not use the expression $e
\oplus t$ for  $e + te_0$ in the current proof. The notation  $x
\oplus t$ is reserved for elements  of $X \oplus_\infty \R$, and $x^*
\oplus \tau$ for elements of  $(X \oplus_\infty \R)^* = X^* \oplus_1
\R$. 

For every $t \in [0,1]$ denote $A_t = (tA - (1-t)A) \cap ( (1-t)A -  tA)$.  Then,
\begin{equation} \label{formula-symm-intersect2} 
W = \bigcup_{t \in [0,1]}\left(A_t \oplus(2t-1)\right).
\end{equation}
Geometrically this means that the lowest level section (with $t = 0$)
of $W$ is the set $(A\cap-A) \oplus(-1) = B_E \oplus(-1) $, when we
move to higher levels the section transforms up to $\frac{A - A}{2}
\oplus 0$ when $t = \frac12$, and then transforms back until
$(A\cap-A) \oplus 1 = B_E \oplus 1 $  when $t = 1$. The set $W$ is not
only centrally  symmetric with respect to zero, but also doubly
mirror-symmetric in the following sense: for every $e \in E$, $a, b
\in \R$, if $(e + ae_0) \oplus b \in W$, then $(\pm e  \pm ae_0)
\oplus (\pm b) \in W$ for all choices of $\pm$.

Let us assume to the  contrary that $W$ is SCD. From this assumption
we are going to deduce that  $B_E$ is SCD, which will be the desired
contradiction. Let $S_n=S(W, w_n^*, \eps_n)$ form a determining
sequence of slices of $W$,  $w_n^* = x_n^* \oplus \tau_n$.  Denote
also  $e_n^* \in E^*$ and $s_n \in \R$ those elements that represent
the corresponding $x_n^*$, i.e.,  $x_n^*(e + te_0)= e_n^*(e) + s_n t $
for all  $e \in E$, $t \in \R$. By the Bishop-Phelps theorem the set of 
functionals that attain their supremum on $W$ is norm-dense in the dual space, 
consequently, by a small perturbation argument,  we may
assume that each $w_n^*$ attains its supremum $R_n$ on  $W$ 
at some point $w_n = x_n \oplus b_n= (e_n + a_ne_0) \oplus b_n \in W$, 
$b_n = 2t_n - 1$, that is 
$$
R_n := \sup_{w \in W} w_n^*(w) =  x_n^*(x_n) + \tau_n b_n = e_n^*(e_n)
+ s_n a_n + \tau_n b_n. 
$$
We are going to show that  $\tilde S_n = S(B_E, e_n^*, \eps_n)$, $n
\in \N$, form  a determining sequence of slices of $B_E$. Fix an
arbitrary $e^* \in S_{E^*}$ and $\eps \in (0, 1)$. According to
Lemma~\ref{len-H-B-scd}, our task is to find an  $n \in \N$ such that
$\tilde S_n \subset S(B_E, e^*, \eps)$. Let us extend $e^*$ to the
whole  $X \oplus_\infty \R$ by the natural rule $e^*((e +
t_1e_0)\oplus t_2) := e^*(e)$ and consider the corresponding slice
$S(W, e^*,  \frac{\eps}{2})$. Due to the same  Lemma~\ref{len-H-B-scd}
there is an $m \in \N$ such that  $ S_m \subset S(W, e^*,
\frac{\eps}{2})$. Remark  that the corresponding $e_m^*$ is non-zero,
otherwise with every point $ (e + c e_0)\oplus d$ the slice $S_m$
would contain also $ ( ce_0)\oplus d$, thus contradicting the
inclusion $ S_m \subset S(W, e^*,  \frac{\eps}{2})$. 
Without loss of generality we may assume that $s_m,
\tau_m \ge 0$ (here we use the symmetry of $W$ and of $S(W, e^*,
\frac{\eps}{2})$ with respect to corresponding changes of signs). Then
we can also assume $a_m, b_m \ge 0$ and consequently $t_m \ge
\frac12$. 

By the definition, $x_m^*(x_m) = \sup x_m^*\left( A_{t_m}  \right)$.
We claim  that  in fact 
\begin{equation} \label{formula-xm*xm} 
x_m^*(x_m) = \sup x_m^*\left( (1 - t_m) A + t_m B_E  \right)= (1 -
t_m)  \sup x_m^*(A) + t_m \|e_m^*\|. 
\end{equation}
Indeed, $A_{t_m}  = ({t_m} A - (1-{t_m} )A) \cap ( (1-{t_m} )A -
{t_m} A) \subset (1-{t_m} )A -  {t_m} A$, so   $x_m \in A_{t_m}$ has a
representation of the form $x_m =  (1-{t_m} ) y - t_m z$ with $y, z
\in A$. Consequently,  
\[
\begin{split}
x_m^*(x_m) & =   (1-{t_m}) x_m^*(y)  + t_m  x_m^*(-z) \\
& \le (1 - t_m) \sup x_m^*(A)  + t_m \sup x_m^*(-A) \\
&=  (1 - t_m)  \sup x_m^*(A) + t_m \|e_m^*\|,
\end{split}
\]
where we used   the positivity of $s_m$ in the last step. For the
reverse inequality in \eqref{formula-xm*xm}  we can use the inclusion
$ A \supset B_E$, the inequality $t_m \ge 1 - t_m$ and
Lemma~\ref{lem-convhull-monotonicity} which together give us the
inclusion $ (1-{t_m} )A + {t_m} B_E \subset {t_m} A + (1-{t_m} )
B_E$. This implies that   
\begin{align}  \label{inclusion-Atm} 
\nonumber \qquad \qquad A_{t_m}   &= ({t_m} A - (1-{t_m} )A) \cap (
 (1-{t_m} )A -  {t_m} A)  \\ 
 & \supset  ({t_m} A - (1-{t_m} )B_E) \cap ( (1-{t_m} )A -  {t_m} B_E)  \\ 
\nonumber &=  (t_m A + (1-{t_m} )B_E) \cap ( (1-{t_m} )A +  {t_m} B_E) \\
\nonumber & \supset  (1 - t_m) A + t_m B_E, 
\end{align}
 so 
$$
x_m^*(x_m) = \sup x_m^*\left( A_{t_m}  \right) \ge  \sup x_m^*\left(
  (1 - t_m) A + t_m B_E  \right). 
$$
Thus, the formula \eqref{formula-xm*xm}  is proved. It remains to
prove that  $\tilde S_m \subset S(B_E, e^*, \eps)$, or in other words
that $\tilde S_m \setminus S(B_E, e^*, \eps) = \emptyset$.  Assume
that this set is not empty, and pick an arbitrary $e \in \tilde S_m
\setminus S(B_E, e^*, \eps) $. Then $e \in B_E$ and $e$ satisfies
simultaneously two inequalities: 
\begin{equation} \label{dva-uslovija} 
e_m^*(e) > \|e_m^*\| - \eps_m, \quad \textrm{ and } \quad e^*(e) \le 1 - \eps.
\end{equation}
Take an arbitrary $g \in A$ with $x_m^*(g) > \sup x_m^*(A) -
\eps_m$. According to \eqref{inclusion-Atm}, $(1 - t_m) g + t_m e \in
A_{t_m} $, so    
$$
((1 - t_m) g + t_m e)   \oplus  b_m  = ((1 - t_m) g + t_m e)   \oplus
(2 t_m - 1)\in W. 
$$  
Then, the following inequality
\[
\begin{split}
w_m^*(((1 - t_m) g + t_m e)  \oplus  b_m) 
&=   x_m^*((1 - t_m) g + t_m e) + \tau_m b_m \\
& > (1 - t_m)( \sup x_m^*(A) -   \eps_m) + t_m (\|e_m^*\| -\varepsilon_{m})  + \tau_m b_m \\
&= x_m^*(x_m) + \tau_m b_m  - \eps_m =R_m - \eps_m
\end{split}
\]
implies that    $((1 - t_m) g + t_m e)   \oplus  b_m \in S_m $, and
consequently  $((1 - t_m) g + t_m e)   \oplus  b_m \in  S(W, e^*,
\frac{\eps}{2})$. This means that  
$$
(1 - t_m)e^*(g) +  t_me^*(e)   =  e^*(((1 - t_m) g + t_m e)  \oplus
b_m) > 1 -    \frac{\eps}{2}. 
$$
Together with the second condition from \eqref{dva-uslovija} this gives
$$
1 -    \frac{\eps}{2} < (1 - t_m)e^*(g) +  t_me^*(e) <  (1 - t_m) +
t_m(1 - \eps) = 1 - t_m \eps \le 1 -    \frac{\eps}{2}. 
$$
This contradiction proves that $\tilde S_m \setminus S(B_E, e^*, \eps)
= \emptyset$.   
\end{proof}


\section{Relationship between SCD and aSCD sets}\label{sec4}

In this section we prove two main results that answer \cite[Question
7.5]{SCDsets}. Namely, we demonstrate that  the properties  aSCD and
SCD are not equivalent for general bounded closed convex sets, but in
the most important case for the applications, namely that  of balanced
bounded closed convex sets, the equivalence holds true. 

\subsection{aSCD and SCD are not equivalent}

In order to present the promised example  of an aSCD set that is not
SCD, we need some more preparatory work. We keep the  notation from
Proposition~\ref{prop:counterexample}. In particular,  $X$  is the
direct sum of its subspace $E$ and a one-dimensional subspace $ \spn
e_0$ equipped with an equivalent norm in which it can be identified
with $E \oplus_{\infty} \mathbb{R}$,  and $E$ possesses the Daugavet
property.

\begin{Lemm}\label{Lemm:counterexample2}
Denote 
\[ 
\begin{split}
C_1  &=  \left\{ y \oplus t \in X \colon t \in \left[0, \frac{1}{\sqrt
      3}\right], \, \| y\| \le  \sqrt\frac{1 -  3 t^{2}}{1 + t^2}
\right\}, \\ 
C_2  &=   \left\{ y \oplus t  \in X \colon t \in \left[0,
      \frac{1}{\sqrt 3}\right], \, \| y\| \le \sqrt{1 - 3 t^{2}}
      \right\}. 
\end{split}
\]
Then for every subset $C \subset X$ satisfying $ C_1 \subset C \subset C_2$ and
for every slice $S$ of $C$ we can find another slice $S' = S(C,
y^{\ast} \oplus \lambda, \varepsilon) \subset S$ with $\lambda > 0$
and $S' \cap B_{E} = \emptyset$.  
\end{Lemm}

\begin{proof}
 We are going to prove two claims from which we can deduce the result easily.

\textbf{Claim 1}: \emph{Every slice of $C$ contains a slice of the
  form $S(C, x^{\ast}, \varepsilon)$ where $x^{\ast} = y^{\ast} \oplus
  \lambda$ with $\lambda > 0$}.

Let $S=S(C, x^{\ast}, \varepsilon)$ be a slice of $C$ with $x^{\ast} =
y^{\ast} \oplus \lambda \neq 0$, and $\lambda \le 0$.  We are going to
distinguish three cases: 

\noindent \textbf{(A)}  $\lambda = 0$. Without loss of generality we
can take $\| y^{\ast}\|= 1$, and consequently $\sup\{x^{\ast}(a): a
\in C\} = 1$.   In this case  $S(C, y^{\ast} \oplus (\varepsilon/2),
\varepsilon/2)$ satisfies that each of its elements $x \oplus t$ has
the property 
\[ 
y^{\ast}(x) + \frac{\varepsilon}{2} \ge \ y^{\ast}(x) + t
\frac{\varepsilon}{2} \ge \sup\left\{\langle y^{\ast} \oplus
  \frac{\varepsilon}{2}, a \rangle: a \in C \right\} -
\frac{\varepsilon}{2} \ge 1 - \frac{\varepsilon}{2}, 
\] 
that is $y^{\ast}(x)  \ge 1 - \eps$. 
Therefore $S(C, y^{\ast} \oplus (\varepsilon/2), \varepsilon/2)
\subset S(C, x^{\ast}, \varepsilon)$. 

\noindent  \textbf{(B)} $ y^{\ast} = 0$, $\lambda < 0$.  Without loss
of generality we can take $\lambda = - 1$. In this case 
\[ 
S(C, x^{\ast}, \varepsilon) = \left\{x \oplus t \in C: t \in
  [0,\eps)    \right\}.
\] 
Take an arbitrary $ e^{\ast} \in S_{E^*}$. Taking into  account that
all elements $x \oplus t \in C$ satisfy  $ t  \le \sqrt\frac{1 -
  \|x\|^{2}}{3}$, we obtain that  for all $ \delta < \eps^2$  
\[ 
\begin{split}
S(C, e^{\ast} \oplus 0, \delta) &=  \left\{x \oplus t \in C: e^*(x) >
  1 - \delta   \right\} \\  
&\subset \left\{x \oplus t \in C:  \|x\| > 1 - \delta   \right\} \\
&\subset S(C, x^{\ast}, \varepsilon),
\end{split}
\]
which reduces the problem to the case (A).

\noindent  \textbf{(C)}\,  $ y^{\ast} \neq 0$, $\lambda < 0$. Again,
without loss of generality we can take $\| y^{\ast}\|= 1$, and  since
$\lambda$  is negative,  $\sup\{x^{\ast}(a): a \in C\} = \| y^{\ast}\|
= 1$. Fix an $\eta > 0$ small enough to have  $\eta^{2} - \lambda \eta
< \eps$.  In this case the slice $S(C, y^{\ast} \oplus 0, \eta^{2})$
satisfies that each of its elements $x \oplus t$ has the property 
\[ 
\sqrt{1 - 3t^{2}} \ge y^{\ast}(x) > 1- \eta^{2} 
\]
which yields that $t < \eta$ and so
\[ 
\langle y^{\ast} \oplus \lambda, x \oplus t \rangle = y^{\ast}(x) +
\lambda t \ge 1 - \eta^{2} + \lambda \eta> 1 - \eps. 
\] 
We obtain that 
$$
S(C, y^{\ast} \oplus 0, \eta^{2}) \subset S(C, y^{\ast} \oplus
\lambda, \varepsilon), 
$$ 
which again  reduces the problem to the case (A).

\textbf{Claim 2}:  \emph{For every $x^{\ast} = y^{\ast} \oplus \lambda
  \in X^{\ast}$ with $\lambda > 0$ there is $\varepsilon > 0$ such
  that } 
$$
S(C, y^{\ast} \oplus \lambda, \varepsilon) \cap B_{E} = \emptyset.
$$

 If $y^{\ast} = 0$  then the slice  $S(C, 0 \oplus \lambda,
 \frac{\lambda}{2\sqrt 3})$ does not intersect $ B_{E}$. It remains to
 consider  $y^{\ast} \neq 0$.  Without loss of generality we can
 assume that $\| y^{\ast}\| = 1$. Then 
\[ 
\begin{split}
\sup\{x^{\ast}(a): a \in C\}  
& \ge \sup\{x^{\ast}(a): a \in C_1\} \\ 
&= \sup_{t \in [0, \frac{1}{\sqrt{3}}]}\, { \sup_{\| x\| \le
    \sqrt{\frac{1 - 3t^{2}}{1+t^{2}}}}\{y^{\ast}(x) + \lambda t}\} \\ 
& = \sup_{t \in [0, \frac{1}{\sqrt{3}}]}\left\{\sqrt{\frac{1 -
      3t^{2}}{1+t^{2}}} + \lambda t\right\}.  
\end{split}
\]
Denote  $g(t)=\sqrt{\frac{1 - 3t^{2}}{1+t^{2}}} + \lambda t$. 
It is standard to check that $g'(0) = \lambda> 0$ so the  supremum of
$g(t)$ on  $ [0, \frac{1}{\sqrt{3}}]$ is strictly greater than $ g(0)
= 1$. We can then write this  supremum as $1 + \delta_{\lambda}$ for
some $\delta_{\lambda} > 0$. Finally, we are going to check that $S(C,
y^{\ast} \oplus \lambda, \delta_{\lambda} / 2)$ satisfies the desired
property: for each $x  \in B_{E}$ we have that  
\[ 
\langle y^{\ast} \oplus \lambda, x \rangle =  y^{\ast}(x)  \le 1 <
1 + \frac{\delta_{\lambda}}{2} \le  \sup\{x^{\ast}(a): a \in C\} -
\frac{\delta_{\lambda}}{2}. 
\] 
This finishes the proof of Claim~2 and of our Lemma.
\end{proof}

Remark  that according to (a) and (b) of
Proposition~\ref{prop:counterexample}, the set $A$ satisfies the
condition $C_1 \subset A \subset C_2$ 
of Lemma~\ref{Lemm:counterexample2}. Since $A$ is SCD, this
implies the following corollary. 

\begin{Prop}\label{Prop:counterexample2}
Let $X$ and $A$ be as in Proposition~\ref{prop:counterexample}. Then,
we can choose a determining sequence of slices $S_{n}= S(A,
x_{n}^{\ast}, \varepsilon_{n})$   of $A$ such that $x_{n}^{\ast} =
y_n^{\ast} \oplus \lambda_{n}$ with $\lambda_{n} > 0$ and moreover
$S_{n} \cap  E  = \emptyset$. 
\end{Prop}

Now we are ready for the main result of the subsection that answers
\cite[Question~7.5]{SCDsets} in the negative. 

\begin{Theo}\label{Theo:aSCDcounterexample}
Let $X$ be a Banach space with the Daugavet property. There is a
convex and closed set $\tilde C \subset X$ which is aSCD but not SCD. 
\end{Theo}

\begin{proof}
As before, we reduce the situation to the separable case, consider a
1-codimensional subspace $E \subset X$, and write the whole space as
the direct sum   $X = E \oplus \mathbb{R}e_0$. Let $A$  be the same
set as before: 
$$
A  = \{ x \oplus t \in X \colon \| x\|_{t}^{2} + 3t^{2} \le 1, \ t \ge 0 \}.
$$
 Let $\alpha := 1/(2 \sqrt{3})$. Select  $\delta > 0$ small enough
 so that $ \delta B_E \oplus \alpha \subset A$ and denote $C =
 \cconv\left(B_E \cup ( \delta B_E  \oplus\alpha)\right)$. 
 We are going to consider the set
\[ 
\tilde C:= A \cup (-C). 
\]
\begin{itemize}
\item $\tilde C$ is convex and closed: $A$ and $-C$ are convex and
  closed sets, and it is easy to deduce that so is  $\tilde C$ using
  that $\tilde C \subset B_{E} \oplus \mathbb{R}$ and $A \cap (-C) =
  B_{E}$. 
For the same reason, 
\begin{equation} \label{eq-tildaC=}
\tilde C = \cconv(A \cup (- (\delta B_E  \oplus\alpha))).
\end{equation}
\item $\tilde C$ is aSCD: We know by
  Proposition~\ref{prop:counterexample} that $A$ is SCD. If
  $(S_{n})_{n \in \N}$ is a determining sequence of slices for $A$,
  then by Proposition~\ref{Prop:counterexample2} we can also assume
  that $S_{n} = S(A, y_{n}^{\ast} \oplus \lambda_{n},
  \varepsilon_{n})$ with $\lambda_{n} > 0$ and $S_{n} \cap B_{E} =
  \emptyset$. This yields that  $(S_{n})_{n \in \N}$ is actually a
  sequence of slices of $\tilde C$. It moreover satisfies the
  condition of aSCD for $\tilde C$: given $x_{n} \in S_{n}$ we have
  that $A \subset \overline{\conv}{\{ x_{n}\colon n \in \N\}}$ by the
  choice of the sequence, so 
\[ 
A \cup (- C) \subset A \cup (-A) \subset \overline{\aconv}{\{ x_{n}
  \colon n \in \N \}}. 
\]
\item $\tilde C$ is not SCD: Assume that $\tilde C$ is SCD. Then from
  \eqref{eq-tildaC=} we deduce that  $A \cup (- (\delta B_E
  \oplus\alpha) )$ is SCD. A direct application of
  Lemma~\ref{lem-slice-hyperplane++} gives us that $\delta B_E
  \oplus\alpha$   is SCD, which is impossible, because this set is 
  obtained from the non-SCD set $\delta B_E$ by a shift.  
\end{itemize} 
\end{proof}

\subsection{The case of  balanced sets}

Since both the definition of a balanced set and of an aSCD set depend
on the scalar field, in this subsection we address both the cases of
real  and complex scalars. 

 Let $U \subset X$  be a convex closed bounded balanced set, $x^* \in
 X^*$ be a non-zero functional, $\eps > 0$. Denote $S^b(U, x^*, \eps)
 $ the corresponding \emph{balanced slice} of $U$: 
 \[ 
 S^b(U, x^*, \eps)  = \Bigl\{ x \in U \colon  |x^{\ast}(x)| > \sup_{a
 \in U}  |x^{\ast}(a)| - \varepsilon \Bigr\}. 
 \]
Also, let us call a sequence of sets $V_n \subset U$, $n\in\N$,
\emph{balanced determining} for $U$ if for each $B \subset X$ that
intersects all the $V_{n}$,  $n \in \N$,  it holds that  $U \subset
\overline{\aconv}{(B)}$. 

The following proposition is completely analogous to Lemma~\ref{len-H-B-scd}.

\begin{Lemm}  \label{lem-H-B-absscd} 
 Let $U \subset X$  be a convex closed bounded balanced set. Then, the
 following conditions on a sequence $\{V_n: n\in\N\}$ of non-empty
 subsets of $U$ are equivalent: 
\begin{enumerate}
\item[(i)] 
$\{V_n\,:\,n\in\N\}$ is balanced determining for $U$.
\item[(ii)]  
Every balanced slice of $U$ contains one of the $V_n$.
\end{enumerate}
\end{Lemm}

\begin{proof}
(i) $\Rightarrow$ (ii). Assume that some balanced slice $S =  S^b(U, x^{\ast}, \varepsilon)$ of $U$ does not contain any of the
$V_n$. Then $B:=U \setminus S$ intersects all the $V_n$. But $B
\subset  \{ x \in X \colon  |x^{\ast}(x)| \le \sup_{a \in U} \,
|x^{\ast}(a)| - \varepsilon \}$, which is a convex closed balanced
set. Consequently  $\overline{\aconv} B \subset  \{ x \in X \colon
|x^{\ast}(x)| \le \sup_{a \in U} \, |x^{\ast}(a)| - \varepsilon \}$
which means that $\overline{\aconv} B  \not\supset U$, and
consequently $\{V_n\,:\,n\in\N\}$ is not balanced determining. 

(ii) $\Rightarrow$ (i). Let  $B\subset U$ intersect all the
$V_n$.  Then $B$ intersects all the balanced slices of $U$, so for
every  $x^* \in X^*$ we have $\sup_{u \in U} \, |x^{\ast}(u)|  =
\sup_{b \in B} \, |x^{\ast}(b)| $. This means that $B^o = U^o$ and by
the bipolar theorem  
 $U = \overline{\aconv}{(B)}$. 
\end{proof}

\begin{Theo}  \label{theo-aSCD=SCD} 
 Let $U \subset X$  be a convex closed bounded balanced aSCD set, then
 $U$ is SCD.  
\end{Theo} 

\begin{proof}
1.~The real case.
Let $V_n \subset U$, $n\in\N$, form a  balanced determining sequence
of slices for $U$, and let us demonstrate that the $\pm V_n$,
$n\in\N$,  are determining for $U$. According to
Lemma~\ref{len-H-B-scd} we must demonstrate that every  slice $S =
S(U, x^{\ast}, \varepsilon)$ of $U$ contains one of the $ \pm
V_n$. Consider the corresponding balanced slice $S^b =S^b(U, x^*,
\eps) = S \cup (-S)$. Due to Lemma~\ref{lem-H-B-absscd}, there is an
$m \in \N$ such that $V_m \subset S^b$. Since $V_m$ is connected, it
must be contained in one of connected components of  $S^b$, that is
either  $V_m \subset S$, or  $V_m \subset (-S)$. In the first case the
job is done, and in the second case  $-V_m \subset S$, which is also
fine for us.

2.~The complex case. 
Let $V_n \subset U$, $n\in\N$, form a  balanced determining sequence
of slices for $U$, and let $\{r_m\}_{m \in \N}$ be a dense subset of
the unit circle $\mathbb{T}$ of $\C$. We will check now that the $r_m
V_n$, $m,n \in \N$, are determining for $U$. According to
Lemma~\ref{len-H-B-scd} we must demonstrate that every  slice $S =
S(U, \Real x^{\ast}, \varepsilon)$ of $U$ contains one of the $ r_{m}
V_n$. For this, let $\alpha = \sup_{u \in U}{|x^{\ast}(u)|}$ 
(which we assume to be ${>0}$ to avoid a trivial situation) 
and consider the balanced slice 
$S^b = S^b(U, x^*, \eps/2) $. 
Due to Lemma~\ref{lem-H-B-absscd}, there is an 
$n_{0} \in \N$ such that $V_{n_{0}} \subset S^b$. 
This means that 
$$
\{x^*(v)\colon v\in V_{n_0}\} \cap \{z\in\C\colon |z|<\alpha-\eps\}
=\emptyset,
$$
in fact, these convex subsets of $\C=\R^2$ have positive distance and
can hence be strictly separated, by the Hahn-Banach theorem, so that
there is some $r\in \T$ such that 
$$
\inf_{v\in V_{n_0}} \Real r x^*(v)
> \sup_{|z|<\alpha-\eps} \Real r z,
$$
We may assume by a perturbation argument
that $r=r_{m_0}$ for some $m_0$, thus
$$
\Real x^*(r_{m_0}v) > \alpha-\eps \qquad\mbox{for all }v\in V_{n_0},
$$
in other words $r_{m_0}V_{n_0}\subset S$.
\end{proof}


\section{Open problems}

In this small section we list several questions about SCD sets that we
have been 
unable to solve. Some of these questions
were mentioned explicitly in previous papers, some of them appeared
implicitly,  and some of them are motivated by the results of our
paper. 

\begin{enumerate}
\item 
Does every separable Banach space that is not SCD possess the Daugavet
property in some equivalent norm? 
\item 
Does there exist a pair $U_1$, $U_2$ of hereditarily SCD subsets of a
Banach space such that $U_1 + U_2$ is not SCD? 
\item 
Does the relative weak topology on a  closed convex bounded SCD-set $U
\subset X$  always have a countable $\pi$-basis? 
\item 
Is every space with an unconditional basis  SCD?
\item 
Must the union of two hereditarily SCD subsets of a Banach space  be an SCD set?
\end{enumerate}

Concerning the last problem remark that  $\conv(U_1 \cup U_2)$ need
not  be hereditarily SCD when  $U_1$, $U_2$ are hereditarily
SCD. Indeed, if  $U_1$, $U_2$ are the hereditarily SCD sets from
\cite[Corollary~2.2]{SCDsum} whose  Minkowski sum  is not hereditarily
SCD, then $\conv(U_1 \cup U_2) \supset \frac12 (U_1 + U_2)$, so 
$\conv(U_1 \cup U_2)$ is not  hereditarily SCD either.



\begin{thebibliography}{99}

\bibitem{SCDspaces}  {\sc A.~Avil\'{e}s, V.~Kadets, M.~Mart\'{i}n,
    J.~Mer\'{i}, V.~Shepelska}. \emph{Slicely countably determined
    Banach spaces}.  
C. R., Math., Acad. Sci. Paris \textbf{347}, no.~21--22, 1277--1280 (2009).

\bibitem{SCDsets} {\sc A.~Avil\'{e}s, V.~Kadets, M.~Mart\'{i}n,
    J. Mer\'{i}, V.~Shepelska}. \emph{Slicely countably determined
    Banach spaces}. Trans. Amer. Math. Soc.
  \textbf{362}, no.~9, 4871--4900 (2010). 

\bibitem{Bosenko}  {\sc  T.~V.~Bosenko}.
 \emph{Strong Daugavet operators and narrow operators with respect to
 Daugavet centers}. 
Visn. Khark. Univ., Ser. Mat. Prykl. Mat. Mekh. \textbf{931}, No. 62,
 5--19 (2010). 

\bibitem{DeGoZi} {\sc  R.~Deville, G.~Godefroy, V.~Zizler}.
Smoothness and renormings in Banach spaces. Pitman Monographs and
Surveys in Pure and Applied Mathematics. 64. Harlow: Longman
Scientific \& Technical. New York: John Wiley \& Sons, Inc. (1993). 

\bibitem{montesinosAnalysis} 
{\sc M.~Fabian, P.~Habala, P.~H\'ajek, V.~Montesinos, V.~Zizler}.
Banach space theory: the basis for linear and nonlinear
analysis. Springer (2011). 

 \bibitem{LUR}   {\sc M.~I.~Kadets},  \emph{Spaces isomorphic to a
 locally uniformly convex space}.  Izv. Vyssh. Uchebn. Zaved. Mat.,
 \textbf{1959}, no.~6,  51--57  (1959). 

\bibitem{Rearrangement} {\sc V.~Kadets, M.~Mart\'{i}n, J.~Mer\'{i},
    D.~Werner}. \emph{Lushness, numerical index $1$ and the Daugavet
    property in rearrangement invariant
    spaces}. Canad. J. Math. \textbf{65}, 331--348  (2013). 

\bibitem{LipSlices} {\sc V.~Kadets, M.~Mart\'{i}n, J.~Mer\'{i},
    D.~Werner}. \emph{Lipschitz slices and the Daugavet equation for
    Lipschitz operators}.  
Proc. Amer. Math. Soc.  \textbf{143}, no.~12, 5281--5292 (2015).

\bibitem{Spears} {\sc V.~Kadets, M.~Mart\'{i}n, J.~Mer\'{i},
    A.~P\'{e}rez}. \emph{Spear operators between Banach spaces}. 
 \href{https://arxiv.org/abs/1701.02977}{arXiv:1701.02977}, 2017. 
To appear in Lecture Notes in    Mathematics. 

\bibitem{SCDsum} {\sc V.~Kadets and V.~Shepelska}. \emph{Sums of SCD
    sets and their application to SCD operators and narrow
    operators}. Cent. Eur. J. Math.   \textbf{8},
  no.~1,  129--134 (2010). 

\bibitem{KadSSW} \textsc{V.~M.~Kadets, R.~V.~Shvidkoy,
    G.~G.~Sirotkin,  D.~Werner}.  \emph{Banach spaces with the
    Daugavet property}. 
    Trans. Amer. Math. Soc. \textbf{352}, no.~2, 855--873 (2000).
 
\bibitem{KadSW2} \textsc{V.~M.~Kadets, R.~V.~Shvidkoy,   D.~Werner}.
    \emph{Narrow operators and rich subspaces of Banach spaces with
    the Daugavet property}. 
    Studia Math. \textbf{147}, 269--298 (2001). 

\bibitem{MO} \textsc{M.~Mart\'{\i}n and T.~Oikhberg}.
    \emph{An alternative Daugavet property}.
    J. Math. Anal. Appl. \textbf{294}, no.~1, 158--180 (2004). 
 
\bibitem{Schach} \textsc{W.~Schachermayer} 
  \emph{The sum of two Radon-Nikodym-sets need not be a Radon-Nikodym-set}.
Proc. Amer. Math. Soc.  \textbf{95}, 51--57 (1985).

\end{thebibliography}
\end{document}